\documentclass[11pt]{amsart}

\usepackage{lmodern,amsmath,amsthm, amscd, amsbsy, amsfonts}
\usepackage{latexsym,amssymb,mathrsfs,bbm,enumerate,verbatim}
\usepackage{graphicx,microtype}
\usepackage[pagebackref,colorlinks,citecolor=blue,linkcolor=blue,urlcolor=blue,filecolor=blue]{hyperref}
\usepackage{tikz,tikz-cd}
\usepackage[hmargin=3cm,vmargin=3cm]{geometry}

\newcounter{prob}
\title[Improved homological stability after inverting $2$]{Improved homological stability for configuration spaces after inverting $2$}
\author{Alexander Kupers}
\thanks{Alexander Kupers is supported by a William R. Hewlett Stanford Graduate Fellowship, Department of Mathematics, Stanford University, and was partially supported by NSF grant DMS-1105058.}
\author{Jeremy Miller}
\date{\today}

\newcommand{\Q}{\mathbb{Q}}

\newcommand{\Z}{\mathbb{Z}}

\newcommand{\R}{\mathbb{R}}
\newcommand{\m}{\longrightarrow}

\newcommand{\F}{\mathbb{F}}
\newcommand{\G}{\mathcal{G}}

\newcommand{\mr}[1]{{\rm #1}}

\newtheorem{theorem}{Theorem}[section]
\newtheorem{lemma}[theorem]{Lemma}
\newtheorem{proposition}[theorem]{Proposition}

\newtheorem{remark}[theorem]{Remark}
\newtheorem{definition}[theorem]{Definition}

\begin{document}

\begin{abstract}In Appendix A of \cite{Se}, Segal proved homological stability for configuration spaces with a stability slope of $1/2$. With rational coefficients, this was later improved to a slope of $1$ by Church for orientable manifolds and then by Randal-Williams for manifolds of dimension at least $3$. In this note we prove that the stability slope of $1$ holds even with $\Z[1/2]$ coefficients for manifolds of dimension at least $3$. We also clarify some aspects of Segal's proof for topological manifolds.\end{abstract}

\maketitle

\section{Introduction} 

Let $C_k(M)$ denote the configuration space of $k$ unordered distinct points in a manifold $M$ (unless mentioned otherwise, when we say manifold we mean a paracompact Hausdorff topological manifold). That is $C_k(M)$ is the quotient  $(M^k-\Delta)/\Sigma_k$ with $\Delta = \{(m_1,\ldots,m_k)|m_i = m_j \text{ for some $i \neq j$}\}$ the fat diagonal and $\Sigma_k$ the symmetric group acting by permuting the components. If $M$ is a non-compact connected manifold, there are stabilization maps $t \colon C_k(M) \m C_{k+1}(M)$ whose definition we recall in Section \ref{secStab}. In \cite{Mc1}, McDuff proved the following theorem.

\begin{theorem}[McDuff]
Let $M$ be a non-compact connected smooth manifold which is the interior of a compact manifold with boundary. Then for $k$ sufficiently large compared to $i$, the stabilization map induces an isomorphism $t_* \colon H_i(C_k(M)) \m H_i(C_{k+1}(M))$.
\end{theorem}

The phrase ``$k$ sufficiently large'' was later quantified by Segal in Appendix $A$ of \cite{Se}.

\begin{theorem}[Segal]
Let $M$ be a non-compact connected manifold\footnote{Segal does not make clear the exact conditions on the manifold $M$. In this note we give the details necessary to make his proof work in the generality stated here.}. The stabilization map induces an isomorphism $t_* \colon H_i(C_k(M)) \m H_i(C_{k+1}(M))$ for $i \leq k/2$.
\end{theorem}

The case of $\R^2$ was known prior to Segal by the work of Arnol'd in \cite{Ar}. As  $C_1(\R^n) \cong \R^{n}$ and $C_2(\R^n) \simeq \R P^{n-1}$, it is clear that Segal's stability slope of $1/2$ is optimal with $\Z$ coefficients. However, in \cite{RW}, Randal-Williams proved that it is not optimal with rational coefficients for manifolds of dimension at least $3$.

\begin{theorem}[Randal-Williams]
Let $M$ be a non-compact connected manifold of dimension at least $3$, that is the interior of a manifold with non-empty boundary. The stabilization map induces an isomorphism $t_* \colon H_i(C_k(M);\Q) \m H_i(C_{k+1}(M);\Q)$ for $i \leq k$.
\end{theorem}

This result partially improved on the work of Church in \cite{Ch} who obtained a range of $i \leq k-1$ for all orientable manifolds with finitely-generated rational cohomology. Church's and Randal-Williams' theorems can be rephrased as saying that if you invert all primes the stability slope increases from $1/2$ to $1$. We ask the question: If you only invert some primes, how much if at all does the stability slope increase? We prove that a stability slope of $1$ holds after only inverting the prime $2$.

\begin{theorem}
\label{main}
Let $M$ be a non-compact connected manifold of dimension at least $3$. The stabilization map induces an isomorphism $t_* \colon H_i(C_k(M);\Z[1/2]) \m H_i(C_{k+1}(M);\Z[1/2])$ for $i \leq k$. 
\end{theorem}

By considering $C_1(\R^n) $ and $C_2(\R^n)$, we see that the above theorem is optimal in the following sense: a stability slope of $1$ does not hold in dimension $2$ with $\Z[1/2]$ coefficients nor in any dimension greater than $1$ with coefficients in a ring where $2$ is non-zero and not invertible. Though we do not consider affine linear ranges in this paper, our techniques apply to establish improved ranges for configuration spaces of surfaces but with an offset as in \cite{Ch}. In work in progress, Zachary Himes has used these ideas to establish a slope $(p-1)/p$ range for configuration spaces of surfaces with $\mathbb F_p$-coefficients


Our proof is a streamlining of Segal's proof in \cite{Se}. We avoid the need to consider symmetric products by using Cohen's results in \cite{CLM} instead of Nakaoka's results in \cite{N2}. The use of Cohen's calculations is in fact essential as Segal obtains the best result one could hope for given Nakaoka's work as input. Interestingly, the calculations of Cohen were available to Segal at the time and were even cited by him in \cite{Se}. One can also generalize Randal-Williams' proof to obtain our main theorem if one uses Cohen's calculations, see \cite{palmercantero} (in fact, Cantero and Palmer independently proved our result at the same time). It seems harder to use Church's argument in \cite{Ch} to get any torsion or integral information. He proved representation stability for the rational cohomology groups of ordered configuration spaces. In positive characteristic, there is no simple relationship between the cohomology of ordered and unordered configuration spaces.

We became interested in proving Theorem \ref{main} because this result is relevant for \cite{kupersmillerencell}. There the stability range for configuration spaces gives an upper bound for the range of a local-to-global principle for homological stability.

\subsection{Organization}

The organization of the paper is as follows. In Section \ref{secStab} we recall the definition of the stabilization map. We also recall existence results for handle decompositions for topological manifolds and some corollaries. In Section \ref{Rn} we use Cohen's calculations in \cite{CLM} to prove Theorem \ref{main} in the case $M=\R^n$ and in Section \ref{secopen} we leverage this result to prove Theorem \ref{main} in general.

\subsection{Acknowledgments}

We would like to thank Martin Bendersky, S{\o}ren Galatius, Zachary Himes, Martin Palmer, Oscar Randal-Williams and TriThang Tran. Additionally, we thank the anonymous referee for many helpful suggestions and corrections.

\section{Handle decompositions of non-compact manifolds and stabilization maps}\label{secStab} In this section we discuss some technical aspects of the theory of paracompact Hausdorff topological manifolds with the goal of defining stabilization maps and finding nice handle decompositions compatible with these stabilization maps.

In the smooth setting there is an intimate relationship between Morse theory and handle decompositions. In particular one can go from a Morse function to a handle decompositions using the flow along the gradient vector field of the Morse function. For topological manifolds, one can make sense of topological Morse theory (see Section III.3 of \cite{kirbysiebenmann}) and handle decompositions (see Section III.2 of \cite{kirbysiebenmann}), and these are related as expected using TOP gradient-like fields. In particular we quote page 113 of \cite{kirbysiebenmann}: ``Using TOP gradient-like fields, one can carry through for the topological case the elementary discussion of Morse functions as they relate to cobordisms, surgeries, handles etc.''

\begin{lemma}\label{lem_exhaust}Every non-compact connected manifold has an exhaustion by compact manifolds admitting a finite handle decomposition with a single 0-handle.\end{lemma}

\begin{proof}This follows if the manifold is smoothable or we are in a dimension where topological Morse theory works. In particular, one can take a proper smooth or topological Morse function $f \colon M \to \R$ with a global minimum and inductively cancel all the additional critical points of index $0$ that appear. This only involves finitely many modifications of $f$ on $f^{-1}((\infty,n])$, so the result is a Morse function with a single minimum. The relationship between Morse functions and handle decompositions then gives the desired result.

In dimension $\leq 3$, topological manifolds are smoothable by Moise \cite{moise}. In dimension $4$, all non-compact topological manifolds are smoothable by Theorem 1.1 of \cite{quinnsmooth}. In Section III.3 of \cite{kirbysiebenmann}, we learn that topological Morse theory works in dimension $\geq 6$. By the work of Quinn in \cite{quinnendsiii}, these results can be extended to dimension $5$.\end{proof}

In the paper, the term finite complex will be used to describe spaces obtained by consecutively gluing finitely many cells $D^n$ along their boundaries, not necessarily in order of increasing dimension. The dimension of a finite complex is defined to be the highest occurring dimension of a cell.

\begin{lemma}\label{lem_xexists}Let $M$ be the interior of an $n$-dimensional  manifold $\bar{M}$ admitting a finite handle decomposition with a single 0-handle. Then there is closed subspace $X$ of $M$ homeomorphic to an open subset of a finite complex of dimension $\leq n-1$, such that $M \backslash X$ is homeomorphic to $\R^n$.\end{lemma}

\begin{proof}We prove by induction over the number $m$ of handles of $\bar{M}$ that there exists a $\bar{X}$ in $\bar{M}$ that is a finite complex of dimension $\leq n-1$ and such that $\mr{int}(\bar{M} \backslash \bar{X}) \cong \R^n$. Then $X = \bar{X} \cap \mr{int}(M)$. If $m=1$, $\bar{M} = D^n$ and $\bar{X} = \emptyset$, so we are done. Suppose the lemma is true for $m$ handles and let $\bar{M}$ have $m+1$ handles. Then we can write $\bar{M} = \bar{M}' \cup \text{$d$-handle}$ with $d \geq 1$. Splitting $\bar{M}$ at the cocore $C$ of the $d$-handle, we get a manifold $\bar{N}$ which is homeomorphic to $\bar{M}'$. Using the inductive hypothesis, find a $\bar{X}'$ for $\bar{N}$. Denote the image of $\bar{X}'$ after gluing the two copies of $C$ in $\bar{N}$ together by $\bar{Y}$. Now take $\bar{X} = \bar{Y} \cup C$, which is obtained by glueing the cells in $\bar{X}'$ to $C$. We have that $\mr{int}(\bar{M} \backslash \bar{X}) \cong \mr{int}(\bar{N} \backslash \bar{Y}') \cong \mr{int}(\bar{M}' \backslash \bar{X}') \cong \R^n$.\end{proof}

We note that Segal in \cite{Se} claims the results in the above two lemmas without proof. We presume he implicitly restricts attention to smooth manifolds, but it is possible he had in mind a different proof. The proofs given here depend on deep results in topological manifold theory, some of which were not available at the time. We will also use these results to construct the stabilization map, but not after a further technical lemma.

\begin{lemma}\label{lem_uexists}Let $M$ be a non-compact connected manifold with a proper map $f \colon M \m [0,\infty)$. Then $M$ contains a contractible open submanifold $U$ such that the image of the map $f|_U \colon U \m [0,\infty)$ contains $[N,\infty)$ for some $N \geq 0$.\end{lemma}

\begin{proof}Let $U$ be the union of $U_i = M_i \backslash X_i$ obtained by applying the construction of Lemma \ref{lem_xexists} to an exhaustion $\bar{M_i}$ as in Lemma \ref{lem_exhaust} and note that we can pick the $U_i$ compatibly, in the sense that $U_{i+1} \cap M_i = U_i$. Each $U_i$ is homeomorphic to $\R^n$ and hence contractible. As every based map $S^n \m U$ factors over some $U_i$, we have that $\pi_j(U) = 0$ for all $j \geq 0$ and all base points. By \cite{milnorcw}, every manifold has the homotopy type of a CW-complex and hence $U$ is in fact contractible. From the construction it is clear that $U$ has the desired properties.
\end{proof}

We next recall the definition of the stabilization map $t \colon C_k(M) \m C_{k+1}(M)$ for $M$ non-compact and connected.

Configuration spaces have the following functoriality with respect to embeddings. Let $M_i$ be manifolds and $k_i$ be numbers with $\sum k_i=k$. An embedding $e \colon \bigsqcup M_i \m M$  induces a map of configuration spaces $e' \colon \prod C_{k_i}(M_i) \m C_{k}(M)$ defined by applying the embedding $e$ to the location of each point. 

\begin{lemma}\label{lem_stabemb}Let $M$ be non-compact connected manifold, then we can find an embedding $e \colon M \sqcup \R^n \m M$ such that $e|_M$ is isotopic to $\mr{id}_M$. 

Furthermore, this can be chosen so that $M$ has an exhaustion by the interiors of compact manifolds $\bar{M_i}$ admitting a finite handle decomposition with a single 0-handle and $e$ restricts to an embedding $e_i \colon M_i \sqcup \R^n \m M_i$ such that $e_i|_{M_i}$ isotopic to $\mr{id}_{M_i}$.\end{lemma}

\begin{proof}We start by proving the first part of the lemma. If $M$ were smooth, such an embedding $e$ can be obtained using tubular neighborhood $\nu$ of a proper embedding $\gamma \colon [0,\infty) \m M$ as follows. View $[0,\infty)$ as $\{(x_1,x_2, \ldots,x_n) \in \R^n\, |\, x_1 \geq 0, 0=x_2=\ldots =x_n \}$. To get a local model which will be convenient for constructing the desired embedding, note that the tubular neighborhood $\nu$ is diffeomorphic to $N_1([0,\infty)) = \{\vec x \in \R^n|\inf_{\vec y \in [0,\infty)} d(\vec x,\vec y)<1\}$. In this local model, $e \colon N_1([0,\infty)) \sqcup \R^n \to N_1([0,\infty))$ is given as follows: on $\R^n$ we use a diffeomorphism $\R^n \to N_{1/2}(\vec{0}) = \{\vec x \in \R^n|d(\vec x,\vec{0}) < 1/2\}$ and on $N_1([0,\infty))$ we use an embedding which in cylindrical coordinates $(x_1,\rho,\vec{\phi})$ is given by
\[e(x_1,\rho,\vec{\phi}) = (\lambda(x_1,\rho),\vec{\phi})\]
where $\lambda(-,-) \colon (-1,\infty) \times [0,1) \to (-1,\infty) \times [0,1)$ is an embedding that (i) is the identity on points $(x_1,r) \in (-1,\infty) \times [0,1)$ when $r \geq 2/3$ or $x_1 \leq -\sqrt{(2/3)^2-r^2}$, (ii) has image in the subset of points $(x_1,r) \in (-1,\infty) \times [0,1)$ where $r > 3/5$ or $x_1 < -\sqrt{(3/5)^2-r^2}$ and (iii) maps $(-1,\infty) \times \{0\}$ onto itself. See Figure \ref{figemb} for a picture of such an embedding.

\begin{figure}[h]
\includegraphics[width=12cm]{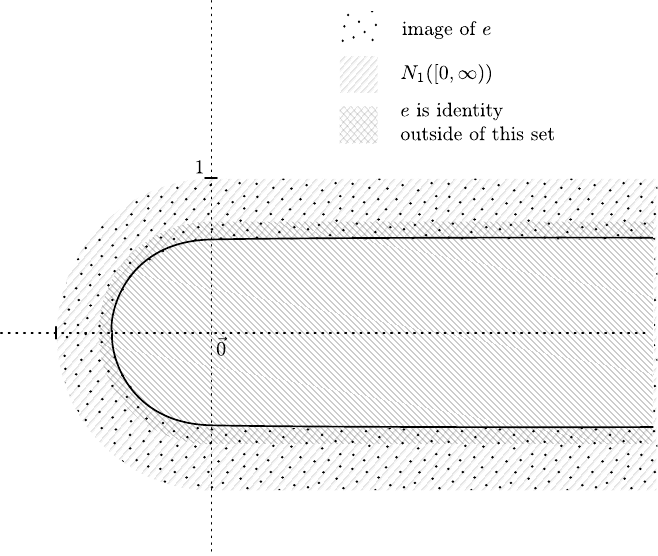}
\caption{The local model for the embedding $e$.}
\label{figemb}
\end{figure}

Note that $e$ is the identity near $\partial  N_1([0,\infty))$ and hence glues to the identity map on the complement of $\nu$ in $M$. In words, $e$ is obtained by pushing in $N_1([0,\infty))$ from infinity, making space of a copy of $\R^n$ near the origin. It is not hard to convince oneself that $e|_{N_1([0,\infty))}$ is isotopic to $\mr{id}_{N_1([0,\infty))}$ through embeddings that are the identity near $\partial  N_1([0,\infty))$.

If $\dim M \geq 5$, $M$ might not be smoothable. If so, take $U$ and $f$ as in Lemma \ref{lem_uexists} and remark that $U$ is smoothable by the Corollary in \cite{lashoficm}. Hence, if we apply the previous construction to an embedding $\gamma \colon [0,\infty) \m U$ such that $f \circ \gamma$ is proper (this is to guarantee that $\gamma$ leaves every compact set in $M$, not just every compact set in $U$), and extend by the identity to $M$, we get the desired embedding.

To obtain the second part of the lemma, we construct a handle decomposition obtained by attaching handles to the closure $\bar{\nu}$ of $\nu$ in $M$. Now apply the techniques of Lemma \ref{lem_xexists} again to $M$, but work relative to a Morse function that coincides on $\nu$ with the function $\sum_{i=1}^n x_i^2$ under the diffeomorphism $\nu \cong N_1([0,\infty))$. The result is a decomposition $\bar{N_i}$ of $M$ obtained by attaching finitely many handles to $\bar{\nu}$, none of which are $0$-handles. We can now take $\bar{M_i} = \bar{N_i} \cup \bar{\nu}_\infty$ where $\bar{\nu}_\infty$ is the compactification of $\bar{\nu} \cong \overline{N_1([0,\infty))} \subset \R^n$ obtained by allowing $x_1$ to have values in $(-\infty,\infty]$ instead of $(-\infty,\infty)$.\end{proof}

\begin{definition}
Fix an embedding $e \colon M \sqcup \R^n \m M$ as in Lemma \ref{lem_stabemb}. Let $T \colon \R^n \times C_k(M) \m  C_{k+1}(M) $ be the map $e'$ precomposed with the natural identification of $\R^n$ with $C_1(\R^n)$. Let $t \colon C_k(M) \m  C_{k+1}(M)$ be given by the formula $t(\vec x)=T(\mathbf 0,\vec x)$ with $\mathbf 0 \in \R^n$ the origin. 
\end{definition}

Note that these maps depend on the choice of the embedding $e$. We would like to stress the use of $\R^n$ in the codomain of $T$. Up to homotopy, it does not matter whether we use $\R^n$ or a point, but in Section \ref{secopen} we will use compactly supported cohomology, which is not a homotopy invariant. 

\section{Homology stability for configurations in $\R^n$}\label{Rn}

Fix an odd prime $p$ and natural number $n>2$. In this section we prove that the stabilization map $C_k(\R^n) \m C_{k+1}(\R^n)$ induces an isomorphism $H_i(C_k(\R^n);\F_p) \m H_i(C_{k+1}(\R^n);\F_p)$ if $i \leq k$. This will imply Theorem \ref{main} when $M =\R^n$.  Before we prove homological stability, we note that all we need to show is surjectivity because of the following result of McDuff (page 103 of \cite{Mc1}).

\begin{theorem}[McDuff]\label{injective}
Let $M$ be a non-compact connected manifold. The stabilization map $t \colon C_k(M) \m C_{k+1}(M)$ is injective on homology. 
\end{theorem}

In fact McDuff's proof makes clear that the map is injective for homology with coefficients in any abelian group. McDuff's theorem is proven by studying the interaction between the stabilization map and the so-called transfer map, whose construction we will sketch now. One cannot define a map $C_k(M) \m C_{k-1}(M)$ by deleting a point since the points are unordered. However, one can define a map $C_k(M) \m (C_{k-1}(M))^k/\Sigma_k$ by deleting a point in all possible ways. For any space $X$, there is a natural map $H_*(X^k/\Sigma_k) \m H_*(X)$ induced by viewing a chain in $X^k/\Sigma_k$ as $k$ chains in $X$ and adding these. The composition of these two maps is called the transfer map. The transfer map is central to the homological stability results of \cite{Mc1}, \cite{Ch} and \cite{RW}, but does not factor as prominently in our proof or in \cite{Se}.

Next we recall homology operations for the homology of configuration spaces. These operations are present in the homology of any algebra over an $E_n$-operad, hence in our case because $\bigsqcup C_k(M)$ is the free $E_n$-algebra on a point. These operations were used by Cohen in \cite{CLM} to explicitly compute $H_*(C(\R^n);\F_p)$. We will use this explicit calculation to prove homological stability.

On page 213 and 217 of \cite{CLM} (also see Section 16.1.2 of \cite{GKRW1}) one finds the following list of homology operations for $p$ odd.\begin{enumerate}
\item A multiplication map: \[\bullet \colon H_q(C_k(\R^n);\F_p) \times H_r(C_j(\R^n);\F_p)  \m H_{q+r}(C_{k+j}(\R^n);\F_p).\]
\item Dyer-Lashof operations: \[Q^s \colon H_q(C_k(\R^n);\F_p) \m  H_{q+2s(p-1)}(C_{pk}(\R^n);\F_p)\]\[ \text{   for } 2s > q \text{ and } 2s-q \leq n-2.\]
\item Dyer-Lashof operations composed with homology Bockstein: \[\beta Q^s \colon H_q(C_k(\R^n);\F_p) \m  H_{q+2s(p-1)-1}(C_{pk}(\R^n);\F_p)\]\[ \text{   for } 2s > q \text{ and } 2s-q \leq n-2.\]
\item A ``top'' operation: \[\xi \colon H_q(C_k(\R^n);\F_p) \m  H_{pq+(n-1)(p-1)}(C_{pk}(\R^n);\F_p)\]\[ \text{   for } n+q \text{ odd}.\]
\item A ``top'' operation composed with a homology Bockstein: \[\beta \xi \colon H_q(C_k(\R^n);\F_p) \m  H_{pq+(n-1)(p-1)-1}(C_{pk}(\R^n);\F_p)\]\[ \text{   for } n+q \text{ odd}.\]
\item A Browder operation: \[\lambda \colon H_q(C_k(\R^n);\F_p) \otimes H_r(C_j(\R^n);\F_p) \m H_{q+r+n-1}(C_{k+j}(\R^n);\F_p).\]
\end{enumerate} 

The operation $\bullet$ is the Pontryagin product associated to an $H$-space structure on the configuration spaces. We write $ab$ for $\bullet(a,b)$. The operations $Q^s$ are the Dyer-Lashof operations and were introduced in \cite{KA} for the prime $2$ and in \cite{DL} for odd primes. The ``top'' operation is often treated as an ``honorary'' Dyer--Lashof operation as it behaves similarly terms of degree shifts, though it has slightly different algebraic properties. Here $\beta$ is the homology Bockstein coming from the short exact sequence $0 \m \mathbb \F_p \m \mathbb Z/p^2\Z \m \F_p \m 0$ of coefficients. The operation $\lambda$ is the Browder operation and was introduced in \cite{Br}. The statement regarding how these operations affect the number of particles is implicit in \cite{CLM} and explicit in Proposition A.4 of \cite{Se} and Section 16.1.2 of \cite{GKRW1}. Note that we exclude the case $q=2s$ where the operation $Q^s$ agrees with the $p$th power map and hence is redundant with the multiplication map. 

Using these homology operations, Cohen calculated the homology of $C(\R^n)$ (see \cite[Theorem 3.1]{CLM}). 

\begin{theorem}[Cohen] \label{explicit}
Fix an odd prime $p$ and let $e \in H_0(C_1(\R^n);\F_p)$ be the class of a point. Let $E=\{e\}$ if $n$ is odd and $\{e,\lambda(e,e)\}$ if $n$ is even. Let $S$ be the set of formal symbols constructed by iterated formal applications of $Q^s$, $\beta Q^s$, $\xi$, and $\beta \xi$ to elements of $E$ (e.g. $Q^3 \beta Q^5 e$ is an element of $S$). Note that we allow zero applications of $Q^s$ and $\beta Q^s$ to elements of $E$ so in particular we consider elements of $E$ to be elements of $S$. There is a subset $G_n \subset S$ such that $H_*(C(\R^n);\F_p)$ is isomorphic as a ring to the free graded commutative algebra on $G_n$. Each element of $G_n$ corresponds to an element of $H_i(C_k(\R^n);\F_p)$ where $i$ and $k$ are computed via the formulas above.

\end{theorem}

Cohen also explicitly described $G_n$  but this will not be needed. Moreover, he computed the homology of configuration spaces of point in $\R^n$ labeled in an arbitrary pointed topological space $X$ with a basepoint relation allowing points to vanish if they are labeled by the basepoint. We are now already ready to prove homological stability for configuration spaces of points in $\R^n$.
\begin{proposition} \label{propRn}
Let $p$ be an odd prime and $n>2$. Then the stabilization map induces an isomorphism $H_i(C_k(\R^n);\F_p) \m H_i(C_{k+1}(\R^n);\F_p)$ if $i \leq k$.

\end{proposition}

\begin{proof}
From the operadic construction of these homology operations in \cite{CLM}, it is clear that $t_* \colon H_i(C_k(\R^n);\F_p) \m H_i(C_{k+1}(\R^n);\F_p)$ is the multiplication by $e$ map. Since the stabilization map is injective (see Theorem \ref{injective}, though it also follows from Theorem \ref{explicit}), it suffices to show that for all $x \in H_i(C_{k+1}(\R^n);\F_p)$ with $i \leq k$, we have that $x=ey$ for some $y \in H_i(C_{k}(\R^n);\F_p) $. We call a homology class $z \in H_i(C_j(M);\F_p)$ \emph{unstable} if $i \geq j$. We will need the following corollaries of how homology operation affect homological degree and number of particles. \begin{enumerate}[(i)]
\item The classes $\beta Q^s e $, $ Q^s e $, $\beta \xi e$, $\xi e$ and $\lambda(e,e)$ are unstable. This is because $2s(p-1) \geq p$, $2s(p-1)-1 \geq p$, $(n-1)(p-1)-1 \geq p$, $(n-1)(p-1) \geq p$, and $n-1 \geq 2$ if $p \geq 3$, $n \geq 3$ and $s \geq 1$.

\item If $z \in H_i(C_{j}(\R^n);\F_p) $ is unstable then $Q^s z$, $\beta Q^s z$, $\xi z$, and $\beta \xi z$ are unstable. This is because $i+2s(p-1) \geq pj$, $i+2s(p-1) -1 \geq pj$, $p i+(n-1)(p-1) \geq pj$ and $p i+(n-1)(p-1)-1 \geq pj$ if $p \geq 3$, $n \geq 3$, $2s \geq i+1$ and $i \geq j$.

\item If $z \in H_i(C_{j}(\R^n);\F_p)$, $w \in H_u(C_{v}(\R^n);\F_p)$ are unstable, then so is $zw$. This is because $i+u \geq j+v$ if $i \geq j$ and $u \geq v$.

\end{enumerate} The first two facts imply that $e$ is the only element of $G_n$ that is not unstable. By the third fact, for a product of elements of $G_n$ to have corresponding number of particles larger than its homological degree, the product must contain a non-zero number of $e$'s. The proposition now follows by Theorem \ref{explicit}.\end{proof}


\begin{remark}
A similar argument also shows that $t_* \colon H_i(C_k(\R^n);\F) \m H_i(C_{k+1}(\R^n);\F)$ is an isomorphism for $i \leq k/2$ with $\F=\Q$ or $\F_p$ with $p$ or $n$ possibly equal to $2$. Likewise, one can show $t_* \colon H_i(C_k(\R^n);\Q) \m H_i(C_{k+1}(\R^n);\Q)$ is an isomorphism for $i \leq k$ provided $n > 2$. 
\end{remark}

\section{Homology stability for configurations in an open manifold} \label{secopen}

In this section we prove homological stability for configuration spaces of particles in a non-compact connected manifold $M$ with an improved range for homology with coefficients in $\Z[1/2]$. Before we prove homological stability, we note that when $M$ is the interior of a manifold with a finite handle decomposition, then the homology of the configuration space is finitely generated. 

\begin{lemma}\label{lem_homfg} Let $R$ be a ring and $M$ be the interior of an $n$-dimensional manifold $\bar{M}$ admitting a finite handle decomposition, then $H_i(C_k(M);R)$ is a finitely-generated $R$-module for all $i \geq 0$.\end{lemma}

\begin{proof} For $M$ smooth, this is proven in the proof of Theorem 4.5 of \cite{Mc1} using the scanning map to a certain space of sections. McDuff's definition of this space of sections only made sense for smooth manifolds as it used the tangent bundle. However, the space of sections can instead be constructed using factorization homology applied to the group completion of $C(\R^n)$, which makes sense for topological manifolds by the work of Ayala and Francis in \cite{Fr2}. Thus, her results can be extended to topological manifolds as well. 

Alternatively, there is the following elementary proof. If $F_k(-)$ denotes the ordered configuration space, there exist fibrations $F_{k-1}(N \backslash \mr{pt}) \m F_k(N) \m N$ for any manifold $N$ without boundary. Using the Serre spectral sequence, this allows one to inductively prove that $F_k(M)$ has finitely-generated $R$-modules as homology groups, since $M$ has finitely-generated $R$-modules as homology groups with coefficients in any local system of finitely generated $R$-modules, using the homotopy equivalence $M \simeq \bar{M}$ and the fact that the latter has a finite handle decomposition. Here we need to remark that if we remove points from the interior of a handle, the fact that $\R^n \backslash \{\text{finitely many points}\}$ is the interior of a manifold with finite handle decomposition implies the same is true for $M \backslash \{\text{finitely many points}\}$. Next there is a fibration $F_k(M) \m C_k(M) \m \Sigma_k$, and since the homology groups of the symmetric group with coefficients in a finitely-generated $R[\Sigma_k]$-module is a finitely-generated $R$-module, another Serre spectral sequence finishes the proof.\end{proof}

Since configuration spaces are manifolds, by Poincar\'e duality, homological stability is equivalent to stability for compactly supported cohomology (with possibly twisted coefficients). The space $C_k(M)$ is orientable if and only if the dimension of $M$ is even (or equal to $1$) and $M$ is orientable. For simplicity we assume that $M$ is even dimensional and orientable; however, one can work with twisted coefficients as is done in Appendix A of \cite{Se} to prove the result for manifolds that are either odd dimensional or non-orientable. In fact, our approach simplifies Segal's in the odd dimensional case, because one does not need to extend the orientation local systems on $C_k(M)$ to a local system on the symmetric product. 

The stabilization map $t \colon C_k(M) \m C_{k+1}(M)$ does not induce a map on compactly supported cohomology. However, the map $T \colon \R^n \times C_k(M) \m C_{k+1}(M)$ is an open embedding and hence induces a map $T_* \colon H^i_c(\R^n \times C_k(M)) \m H^i_c(C_{k+1}(M))$ via extension by zero. By Poincar\'e duality and Theorem \ref{injective}, for orientable even dimensional manifolds, the following proposition is equivalent to Theorem \ref{main}. 

\begin{proposition} Let $M$ be a non-compact connected orientable manifold of dimension $n$ with $n>2$ and even. The map 
\[T_*\colon H^i_c(\R^n \times C_k(M);\Z[1/2]) \m H^i_c(C_{k+1}(M);\Z[1/2])\]
is an isomorphism for $i\geq n(k+1)-k$.

 \label{dual}
\end{proposition}

Compactly supported cohomology is convenient because of the following long exact sequence. One reference is III.7.6 of \cite{iversen}, which uses sheaf cohomology. However Theorem III.1.1 of \cite{bredonsheaf} says that (compactly supported) sheaf cohomology coincides with (compactly supported) singular cohomology for locally path-connected Hausdorff (locally compact) spaces.

\begin{proposition} \label{exact} \label{exactseq} Let $R$ be an abelian group, $Y$ be a locally compact and locally path-connected Hausdorff space and $C \subset Y$ a closed subspace that is also locally path-connected. Let $U = Y \backslash C$ denote its complement. There is a long exact sequence in compactly supported cohomology
\[\cdots \m H^*_c(U;R) \m H^*_c(Y;R) \m H^*_c(C;R) \m H^{*+1}_c(U;R) \m \cdots.\]
The group $R$ can also be replaced by a twisted system of coefficients on $Y$. 
\end{proposition}

Suppose that $M$ is the interior of $\bar{M}$, a compact $n$-dimensional manifold with boundary which admits a finite handle decomposition with a single $0$-handle. Pick $X \subset M$ as in Lemma \ref{lem_xexists}. Let $\mathcal{G}_k^j$ be the subspace of $C_k(M)$ where there are at least $j$ points in $X$. Note that each $\mathcal{G}_k^j$ is closed,  $\mathcal{G}_k^j \supset \mathcal{G}_k^{j+1}$, $\G_k^0=C_k(M)$, $\G_k^{j}$ is empty for $j>k$ and $\mathcal{G}_k^j - \mathcal{G}_k^{j+1}$ is homeomorphic to $C_{k-j}(M \backslash X) \times C_{j}(X)$.

For our choice of embedding, the stabilization map restricts to an open embedding $T\colon \R^n \times \G_k^j \m \G^{j}_{k+1}$ which further restricts to an open embedding: \[T \times \mr{id} \colon \R^n \times C_{k-j}(M \backslash X) \times C_{j}(X) \m C_{k-j+1}(M \backslash X) \times C_{j}(X).\] To prove Proposition \ref{dual} we will need the following proposition in the case $j=0$.

\begin{proposition}
\label{induction}
Let $M$ be the interior of an $n$-dimensional orientable manifold admitting a finite handle decomposition with a single $0$-handle, $X$ as in Lemma \ref{lem_xexists} and suppose $n>2$ and even. Fix an odd prime $p$, then the map $T_* \colon H^i_c(\R^n \times \G_k^j;\F_p) \m H^i_c(\G_{k+1}^{j};\F_p)$  is an isomorphism for $i>n(k+1)-k$ and a surjection for $i=n(k+1)-k$.
\end{proposition}

\begin{proof}
We will first that show that \[(T \times \mr{id})_* \colon H^i_c(\R^n \times C_{k-j}(M \backslash X) \times C_{j}(X);\F_p) \m H^i_c(C_{k-j+1}(M \backslash X) \times C_{j}(X);\F_p)\]   is an isomorphism for $i \geq n(k+1)-k$ and $p$ odd. Recall that $M \backslash X$ is homeomorphic to $\R^n$.  By Proposition \ref{propRn} and Poincar\'e duality, we have that \[T_* \colon H^i_c(\R^n \times C_{k-j}(M \backslash X);\F_p)  \m H^i_c(C_{k-j+1}(M \backslash X);\F_p)\] is an isomorphism for $i \geq n(k-j+1)-k+j$.  Since the dimension of $X$ is $\leq n-1$, $H^i_c(C_{j}( X);\F_p)=0$ for $i > (n-1)j$. This is clearly true if $X$ has one cell, and can be proven inductively by filtering $C_j(X)$ by the number of points in the cell that is attached last.
 
Since we are working over a field, the K\"{u}nneth formula holds (see II.15.2 of \cite{bredonsheaf} for the K\"{u}nneth formula for compactly-supported sheaf cohomology, which coincides with compactly-supported singular cohomology in our setting by the remarks preceding Proposition \ref{exact}). Suppose $x \in H^a_c(C_{k-j+1}(M \backslash X);\F_p)$, $y \in H^b_c(C_{j}( X);\F_p)$ and $x \otimes y$ is not in the image of $(T \times \mr{id})_*$. Since $y \neq 0$, we have $b \leq (n-1)(j)$. Since $x \notin T_*(H^i_c(\R^n \times C_{k-j}(M \backslash X);\F_p) )$, we have that $a < n(k-j+1)-k+j$. Thus, the homological degree of $x \otimes y$ is less than $(k-j+1)n-k+j+(n-1)(j)=n(k+1)-k$. Thus $(T \times \mr{id})_*$ is surjective for $i \geq n(k+1)-k$ with $\F_p$ coefficients. By Theorem \ref{injective}, stabilization maps are always injective on homology. Hence it is an isomorphism for $i \geq n(k+1)-k$.

We now consider the map $T \colon \R^n \times \G_k^j \m \G_{k+1}^{j}$. For $j>k+1$, both spaces are empty and so the map induces an isomorphism on compactly supported cohomology. We now proceed via downward induction on $j$. Suppose the proposition holds for $j+1$. Note that $(\R^n \times \G_k^j) \backslash (\R^n \times \G_k^{j+1})$ is homeomorphic to $\R^n \times C_{k-j}(M \backslash X) \times C_{j}(X)$ and $\G_{k+1}^{j} \backslash  \G_{k+1}^{j+1}$ is homeomorphic to $C_{k-j+1}(M \backslash X) \times C_{j}(X)$. Consider the long exact sequences of Proposition \ref{exact} associated to the inclusions $ \R^n \times \G_k^{j+1} \subset \R^n \times \G_k^j $ and   $  \G_{k+1}^{j+1} \subset \G_{k+1}^{j} $. The stabilization map induces a map between these long exact sequences:

\[\begin{tikzcd}\ldots  \rar &[-10pt] H^i_c( \R^n \times C_{k-j}(M \backslash X) \times C_{j}(X) ;\F_p)  \dar{T} \rar &[-5pt] H^i_c( \R^n \times \G_k^j ;\F_p)  \ar[d]_T \ar[r] &[-5pt] H^i_c( \R^n \times \G_k^{j+1};\F_p)   \dar{T \times \mr{id}} \\
\ldots \rar & H^i_c( C_{k-j+1}(M \backslash X) \times C_{j}(X);\F_p) \rar &H^i_c( \G_{k+1}^{j};\F_p)   \rar & H^i_c( \G_{k+1}^{j+1} ;\F_p).\end{tikzcd}\]  The induction step follows by our induction hypothesis, the previous paragraph and the five lemma. Remark that in the case $i = n(k+1)-k$ the five lemma only permits us to prove a surjection.
\end{proof}


We can now prove the main theorem. 

\begin{proof}[Proof of Theorem \ref{main}] Suppose that $M$ is orientable, even dimensional and the interior of a connected manifold with a finite handle decomposition with only one $0$-handle. Note that Proposition \ref{dual} follows from the $j=0$ case of Proposition \ref{induction}.  Fix an odd prime $p$. By Proposition \ref{dual} and Poincar\'e duality, we have that \[t_* \colon H_i(C_k(M);\F_p) \m H_i(C_{k+1}(M);\F_p)\] is an isomorphism provided $i \leq k$. Let $D_k$ denote the mapping cone of $t \colon C_k(M) \m C_{k+1}(M)$. Since $t_*$ is injective, we have a short exact sequence \[ 0 \m  H_i(C_k(M);\F_p) \overset{t_*}{\m} H_i(C_{k+1};\F_p) \m \tilde{H}_i(D_k;\F_p) \m 0.\] Thus $\tilde{H}_i(D_k;\F_p)$ vanishes for $i \leq k$. View $\F_p$ as a $\Z[1/2]$-module and note that $\Z[1/2]$ is a principle ideal domain. By the universal coefficient theorem, $H_i(C_k(M);\Z[1/2]) \otimes_{\Z[1/2]} \F_p$ injects into $H_i(C_k(M);\F_p)$. This shows that $\tilde{H}_i(D_k;\Z[1/2]) \otimes_{\Z[1/2]} \F_p \cong 0$ for $i \leq k$. By Lemma \ref{lem_homfg}, $H_i(C_{k+1}(M);\Z[1/2])$ is a finitely generated $\Z[1/2]$-module. Since $\tilde{H}_i(D_k;\Z[1/2])$ is a quotient of $H_i(C_{k+1}(M);\Z[1/2])$, it is also a  finitely generated $\Z[1/2]$-module. By the structure theorem for finitely generated $\Z[1/2]$-modules, we see that a finitely generated $\Z[1/2]$-module $A$ is isomorphic to $0$ if and only if $A \otimes_{\Z[1/2]} \F_p \cong 0$ for all odd primes $p$. Thus,  $\tilde{H}_i(D_k;\Z[1/2]) \cong 0$  for $i \leq k$. This shows that $t_* \colon H_i(C_k(M);\Z[1/2]) \m H_i(C_{k+1}(M);\Z[1/2])$ is an isomorphism for $i \leq k$. 

If $M$ is odd dimensional or not orientable, but still the interior of a connected manifold admitting a finite handle decomposition with only one $0$-handle, one modifies the proof by using compactly supported cohomology with coefficients in the orientation systems of the configuration spaces.

Now only assume that $M$ is non-compact and of dimension at least $3$. By Lemma \ref{lem_stabemb}, $M$ has an exhaustion by compact manifolds $\bar{M_i}$ admitting finite handle decompositions with a single 0-handle and this exhaustion can be taken to be compatible with the stabilization maps. We have already established Theorem \ref{main} for the $M_i$. Since homology takes exhaustions by nested open subsets to colimits, one gets a colimit of isomorphisms in the desired range, which is itself an isomorphism.
\end{proof}

We end with two remarks on some further consequences of this proof.

\begin{remark} If the codimension of $X$ is greater than $1$, one can further increase the homological stability slope. Assume $M$ has a subspace $X$ of codimension $q$ with $M \backslash X$ homeomorphic to $\R^n$. Fix an odd prime $p$ and assume $q>1$. In this case, the stability slope with $\F_p$ coefficients is the minimum of $(n-1)/2$ and $(2(p-1)-1)/p$. Note that $(2(p-1)-1)/p<2$. With $\Q$ coefficients,  the stability slope increases to the minimum of $q$ and $(n-1)/2$. Examples of manifolds with such a subspace include $S^q \times \R^{n-q}$. \end{remark}

\begin{remark}

In \cite{palmertwisted}, Palmer proves homological stability for configuration spaces with twisted coefficients. In Remark 1.5 of \cite{palmertwisted} he notes that if one works with twisted systems of $\Q$-vector spaces, the stability slope also increases from $1/2$ to $1$. A similar result is true if one works with twisted systems of $\Z[1/2]$-modules.\end{remark}

\bibliography{cell}{}
\bibliographystyle{alpha}

\end{document}